\documentclass [12pt,reqno]{amsart}
\usepackage{ucs}
\usepackage{amsmath, amssymb, amscd}
\usepackage{pb-diagram}
\usepackage[applemac]{inputenc}
\usepackage{hyperref}
\usepackage[dvipsnames]{xcolor}
\usepackage{tikz-cd}
\usepackage{xypic}
\usepackage{microtype}

\newtheorem{theorem}{Theorem}[section]

\newtheorem{lemma}[theorem]{Lemma}

\newtheorem{corollary}[theorem]{Corollary}
\newtheorem{conjecture}[theorem]{Conjecture}
\newtheorem{proposition}[theorem]{Proposition}

\newtheorem{question}[theorem]{Question}
\theoremstyle{remark}
\newtheorem*{remark}{Remark}
\newtheorem*{outline}{Outline}

\usepackage[top=1.5in, bottom=1.5in, left=1.5in, right=1.5in]{geometry}

\hypersetup{
    colorlinks=true,
    linkcolor=blue,
    filecolor=magenta,      
    urlcolor=NavyBlue,
    pdftitle={Sharelatex Example},
    bookmarks=true,
    pdfpagemode=FullScreen,
}

\renewcommand{\epsilon}{\varepsilon}

\newlength{\bibitemsep}
\newlength{\bibparskip}
\let\oldthebibliography\thebibliography
\renewcommand\thebibliography[1]{%
  \oldthebibliography{#1}%
  \setlength{\parskip}{\bibitemsep}%
  \setlength{\itemsep}{\bibparskip}%
}

\DeclareMathAlphabet{\mathpzc}{OT1}{pzc}{m}{it}
\usepackage{mathrsfs}
\usepackage{tikz-cd}

\newcommand{\Z}{\mathbb{Z}}

\newcommand{\Q}{\mathbb{Q}}

\newcommand{\Hom}{\text{Hom}}

\newcommand{\tmfrac}[2]{\mbox{\large$\frac{#1}{#2}$}}

\title[Rational homology ribbon cobordism is a partial order]{Rational homology ribbon cobordism is a partial order} 

\author{Stefan Friedl}
\address{Stefan Friedl, Universit\"at Regensburg, 93040 Regensburg, Germany}
\email{sfriedl@gmail.com}

\author{Filip Misev}
\address{Filip Misev, Universit\"at Regensburg, 93040 Regensburg, Germany}
\email{filip.misev@mathematik.uni-regensburg.de}

\author{Raphael Zentner}
\address{Raphael Zentner, Universit\"at Regensburg, 93040 Regensburg, Germany}
\email{raphael.zentner@mathematik.uni-regensburg.de}

\begin{document}

\begin{abstract}
   We show that ribbon rational homology cobordism is a partial order within the class of irreducible 3-manifolds. This makes essential use of the methods recently employed by Ian Agol to show that ribbon knot concordance is a partial order. 
\end{abstract}

\maketitle

\thispagestyle{empty}

\section{Introduction}
It has recently been proved by Agol that ribbon concordance of knots is a partial order, \cite{Agol}. In this paper we prove an analogous statement for the preorder on irreducible, closed, connected, oriented 3-manifolds that is given by rational homology cobordisms.

Let $Y_0$ and $Y_1$ be closed, connected, oriented 3-manifolds. We say that a compact, connected, oriented, smooth 4-manifold $W$ is a \emph{cobordism from $Y_0$ to $Y_1$} if $\partial W = Y_1 \sqcup (-Y_0)$. 
A cobordism $W$ is called a \emph{$\Q$-homology cobordism from $Y_0$ to $Y_1$} if the inclusions of $Y_0$ and $Y_1$ into $W$ both induce an isomorphism in homology with $\Q$-coefficients.  Finally we say that $W$ is a \emph{ribbon cobordism from $Y_0$ to $Y_1$} 
if $W$ is built from $Y_0 \times [0,1]$ by attaching only $1$-handles and $2$-handles. 

We write $Y_1\geq Y_0$ if there exists a  ribbon  $\Q$-homology cobordism from $Y_0$ to $Y_1$. Clearly this defines a preorder, i.e. the relation $\geq$ is reflexive and transitive. It is less clear whether this is anti-symmetric and so defines a partial order, i.e. if $Y_0 \geq Y_1$ and $Y_1 \geq Y_0$, are then $Y_0$ and $Y_1$ homeomorphic?
\\

The following conjecture has been formulated by Daemi, Lidman, Vela-Vick, and Wong \cite[Conjecture~1.1]{DLVW}:

\begin{conjecture}\label{conj: rational homology ribbon cobordism conjecture}
	The preorder on the set of homeomorphism classes of closed, connected, oriented 3-manifolds given by ribbon $\Q$-homology cobordism is a partial order, i.e. if one has $Y_0 \geq Y_1$ and $Y_1 \geq Y_0$ then $Y_0$ and $Y_1$ are homeomorphic. 
\end{conjecture}

We prove this conjecture for irreducible 3-manifolds.

\begin{theorem}\label{thm: main theorem}
	The preorder $\geq$ is a partial order on the set of homeomorphism classes of irreducible, closed, connected, oriented 3-manifolds. In particular, if $Y_0 \geq Y_1$ and $Y_1 \geq Y_0$ then $Y_0$ and $Y_1$ are  homeomorphic. 
\end{theorem}

In fact for aspherical 3-manifolds we can prove a refinement.

\begin{theorem}\label{thm: main theorem-2}
	The preorder $\geq$ is a partial order on the set of orientation-preserving homeomorphism classes of aspherical, closed, connected, oriented 3-manifolds. In particular, if $Y_0 \geq Y_1$ and $Y_1 \geq Y_0$ then there exists an orientation-preserving homeomorphism from  $Y_0$ to $Y_1$.
\end{theorem}

It is not clear to us whether the conclusion of Theorem~\ref{thm: main theorem-2} also holds for irreducible 3-manifolds that are not aspherical. This leads us to the following question.

\begin{question}
Does there exist a spherical oriented 3-manifold $Y$ with $Y\leq -Y$?
\end{question} 

In the proofs of Theorems~\ref{thm: main theorem} and~\ref{thm: main theorem-2} we make essential use of the methods employed by Agol to prove that ribbon concordance is a partial order. More precisely, Agol's  methods  apply to the situation we are considering and provide us with the following theorem.

\begin{theorem}\label{thm: ribbon to itself} Suppose $Y$ is a closed, connected, oriented 3-manifold.
		Suppose $W$ is a ribbon $\Q$-homology cobordism from $Y_- \cong Y$ to $Y_+ \cong Y$, i.e. $Y_+ \geq Y_-$ (and so from $Y$ to itself). Then the inclusion $\iota_+ \colon Y_+ \to W$ induces an isomorphism on fundamental groups.
\end{theorem}
Using Theorem~\ref{thm: ribbon to itself}  one can easily prove the following corollary which is the key ingredient in the proofs of Theorems~\ref{thm: main theorem} and~\ref{thm: main theorem-2}.
\begin{corollary}\label{cor:map-on-pi1}
	Suppose $W_0$ is a ribbon $\Q$-homology cobordism from $Y_0$ to $Y_1$, so that $Y_1 \geq Y_0$. Suppose that $W_1$ is a ribbon $\Q$-homology cobordism from $Y_1$ to $Y_0$, so that $Y_0 \geq Y_1$. Then the injection $\iota_0 \colon Y_1 \to W_0$ induces an isomorphism of fundamental groups, and likewise for the injection $\iota_1 \colon Y_0 \to W_1$.  
\end{corollary}
\begin{remark}
		We use the convention of \cite{DLVW}, in which Conjecture \ref{conj: rational homology ribbon cobordism conjecture} has been formulated, where $Y_1 \geq Y_0$ means that there is a ribbon cobordism {\em from} $Y_0$ {\em to} $Y_1$, as defined above. We would like to point out that this may result in some confusion when comparing with Gordon's initial convention in \cite{Gordon} and the one employed by Agol \cite{Agol}, where a ribbon concordance goes from $K_1$ to $K_0$ if the exterior $E(C)$ of the concordance $C$ is obtained by adding only 1-handles and 2-handles to $E(K_0) \times [0,1]$. In other words, in their convention, a ribbon cobordism goes {\em from} the more complex object {\em to} the simpler one. 
\end{remark}

\begin{outline}
	In the second section we will state and prove some auxiliary results that we will need for the proof of our main results, Theorems \ref{thm: main theorem} and~\ref{thm: main theorem-2}.
 In the third section we provide the proof of Theorem \ref{thm: main theorem-2}.
Afterwards in the fourth section we will use 
Theorem \ref{thm: main theorem-2} as an ingredient in our proof of Theorem \ref{thm: main theorem}. Finally the fifth section contains a sketch of proof of Theorem \ref{thm: ribbon to itself} which is almost verbatim identical to Agol's proof in the context of knot complements, together with the proof of Corollary \ref{cor:map-on-pi1}. 
\end{outline}

\subsection*{Acknowledgments} 
SF and FM were supported by the SFB 1085 ``higher invariants'' at the University of Regensburg, funded by the DFG. RZ is thanking the DFG for support through the Heisenberg program. Furthermore, we wish to thank Ederson Dutra and Lukas Lewark for helpful conversations, and Lukas Lewark for comments on a preliminary version of this article. Also we would like to thank Brendan Owens for pointing out a misleading typo in the proof of Theorem \ref{thm: main theorem}

\section{Preparations}\label{sec: preparations}
In this short section we collect a few basic facts that we will use in the proof of 
Theorem \ref{thm: main theorem}. The following appears as \cite[Proposition 2.1]{DLVW}.

\begin{proposition}[Daemi, Lidman, Vela-Vick, Wong]\label{prop:basics-ribbon-cobordism}
Let $Y_0$ and $Y_1$ be closed, connected, oriented 3-manifolds
and let $W$ be a cobordism from $Y_0$ to $Y_1$.
We denote by $\iota_0\colon Y_0\to W$ and $\iota_1\colon Y_1\to W$ the obvious inclusion maps. 
\begin{enumerate}
\item If $W$ is a ribbon cobordism, then the map  $(\iota_{1})_*\colon \pi_1(Y_1)\to \pi_1(W)$ is an epimorphism.
\item If $W$ is a ribbon $\Q$-homology cobordism, then  the map  $(\iota_{0})_*\colon \pi_1(Y_0)\to \pi_1(W)$ is a monomorphism.
\end{enumerate}
\end{proposition}
By definition, a ribbon cobordism $W$ from $Y_0$ to $Y_1$ is obtained from $Y_0\times [0,1]$ by attaching 1-handles and 2-handles. Flipping $W$ upside down, we see that $W$ can equivalently be viewed as $Y_1\times [0,1]$ with some 2-handles and 3-handles attached. This immediately implies the first statement. For the second statement, Daemi, Lidman, Vela-Vick, and Wong use Theorem 2 in Gerstenhaber-Rothaus' work \cite{Gerstenhaber-Rothaus_1962}, much in the same way as initially done by Gordon in \cite[Proof of Lemma 3.1]{Gordon}, using the residual finiteness of 3-manifold groups. 
\\


The next proposition gives us some homological information about $\Q$-homology cobordisms.

\begin{proposition}\label{prop:basics-ribbon-cobordism-h}
Let $Y_0$ and $Y_1$ be closed, connected, oriented 3-manifolds. We consider their fundamental classes $[Y_0]\in H_3(Y_0;\Z)$ and $[Y_1]\in H_3(Y_1;\Z)$. 
Let $W$ be a cobordism from $Y_0$ to $Y_1$.
 If $W$ is a $\Q$-homology cobordism, then the maps  $(\iota_{0})_*\colon H_3(Y_0;\Z)\to H_3(W;\Z)$ and $(\iota_{1})_*\colon H_3(Y_1;\Z)\to H_3(W;\Z)$ are both isomorphisms. Furthermore $(\iota_0)_*([Y_0])=(\iota_1)_*([Y_1])\in H_3(W;\Z)$. 
\end{proposition}

\begin{proof}
 Let $i\in \{0,1\}$. We have the following commutative diagram:
\[ \xymatrix@C0.5cm@R0.65cm{ \dots \ar[r]&H_4(W,Y_i;\Z)\ar[r] & H_3(Y_i;\Z)\ar[rr]^{(\iota_i)_*}\ar[d]&&
H_3(W;\Z)\ar[r]\ar[d] & H_3(W,Y_i;\Z)\ar[r]  &\dots \\&
 & H_3(Y_i;\Q)\ar[rr]^{(\iota_i)_*}_\cong&&
H_3(W;\Q).}\]
Since $W$ is a $\Q$-homology cobordism we know that the bottom horizontal map is an isomorphism. It follows from this observation and the universal coefficient theorem 
that the kernel and the cokernel of the map $(\iota_i)_*\colon H_3(Y_i;\Z)\to H_3(W;\Z)$ are finite.
But by Poincar\'e-Lefschetz duality and the universal coefficient theorem we see that
\begin{equation*}\begin{split}
H_3(W,Y_i;\Z)\cong H^1(W,Y_{1-i};\Z)\cong & \operatorname{Hom}(H_1(W,Y_{1-i};\Z),\Z)\\ & \oplus \operatorname{Ext}(H_0(W,Y_{1-i};\Z),\Z)
\end{split}
\end{equation*}
 is torsion-free. Similarly, we see that $H_4(W,Y_i;\Z)$ is torsion free, since  
 \begin{equation*}\begin{split}
H_4(W,Y_i;\Z)\cong H^0(W,Y_{1-i};\Z)\cong & \operatorname{Hom}(H_0(W,Y_{1-i};\Z),\Z).
\end{split}
\end{equation*}
In summary, we see that the two maps  $(\iota_{0})_*\colon H_3(Y_0;\Z)\to H_3(W;\Z)$ and $(\iota_{1})_*\colon H_3(Y_1;\Z)\to H_3(W;\Z)$ are isomorphisms.

It remains to show that  $(\iota_0)_*([Y_0])=(\iota_1)_*([Y_1])\in H_3(W;\Z)$. 
We consider the long exact sequence of the pair $(W,Y_0\cup Y_1)$. 
\[ \dots \,\to \,H_4(W,Y_0\cup Y_1;\Z)\,\xrightarrow{\, \partial \,} H_3(Y_0\cup Y_1;\Z) \, \to\, H_3(W;\Z)\,\to\,\dots \]
It is well-known that $\partial([W])=[\partial W]$. Since $W$ is a cobordism we have $\partial W=Y_1\sqcup(-Y_0)$, in particular $[\partial W]=[Y_1]-[Y_0]$.
Since the sequence is exact we see that the image of $[Y_1]-[Y_0]$ is zero in $H_3(W;\Z)$, i.e.\ we have $(\iota_0)_*([Y_0])-(\iota_1)_*([Y_1])=0\in H_3(W;\Z)$. 
\end{proof}

\begin{lemma}\label{lemma: commutative diagram}
	Suppose that we have a continuous map $g\colon X \to Z$, between path-connected topological spaces $X$ and $Z$ which admit the structure of a CW-complex. Then for any $k$ the following diagram commutes:
\begin{equation*}
\begin{tikzcd}
	H_k(X;\Z) \arrow{d}[swap]{(j_X)_*} \arrow{r}{g_*} & H_k(Z;\Z)  \arrow{d}{(j_Z)_*}  \\
	H_k(\pi_1(X);\Z) \arrow{r}{g_*} & H_k(\pi_1(Z);\Z).
\end{tikzcd} 
\end{equation*}
Here the group homology $H_i(G;\Z)$ of a group $G$ is defined to be the singular homology of its associated $K(G,1)$-space: $H_i(G;\Z):= H_i(K(G,1);\Z)$. 
Furthermore  $j_X\colon X \to K(\pi_1(X),1)$ and $j_Z\colon Z \to K(\pi_1(Z),1)$ are the natural maps of the respective spaces to the Eilenberg-MacLane space associated to their fundamental groups. 
\end{lemma}
\begin{proof}
	We denote by $K(g)\colon K(\pi_1(X),1) \to K(\pi_1(Z),1)$ the map induced by the group homomorphism $g_* \colon \pi_1(X) \to \pi_1(Z)$. By construction this map induces the map 
	\begin{equation*}
g_* \colon \pi_1(X) \xrightarrow{\cong} \pi_1(K(\pi_1(X),1)) \to \pi_1(K(\pi_1(Z)),1) \xleftarrow{\cong} \pi_1(Z). 
	\end{equation*}
Therefore, the two maps $K(g) \circ j_X$ and $j_Z \circ g$ induce the same maps on fundamental group. Since their image space $K(\pi_1(Z),1)$ is aspherical, these two maps are homotopic by Whitehead's theorem, and hence induce the same maps in homology. 
\end{proof}

\begin{theorem}\label{thm:iso-on-pi1-implies-homeo}
Let $Y_0$ and $Y_1$ be irreducible, closed, orientable  3-manifolds and let
$\alpha\colon \pi_1(Y_0)\to \pi_1(Y_1)$ be an isomorphism.
\begin{enumerate}
\item  If $Y_0$ and $Y_1$ are not lens spaces, then $Y_0$ and $Y_1$  are homeomorphic.
\item If  $Y_0$ and $Y_1$ are not  spherical, then there exists a homeomorphism $g\colon Y_0\to Y_1$ which
induces $\alpha$, i.e.\ which satisfies $g_*=\alpha\colon \pi_1(Y_0)\to \pi_1(Y_1)$. 
\end{enumerate}
\end{theorem}

\begin{proof}
The theorem is stated as \cite[Theorem~2.1.2]{AFW}.
It is a consequence of  the Geometrization Theorem, the
Mostow Rigidity Theorem, work of Waldhausen \cite[Corollary 6.5]{Waldhausen68}, Scott
\cite[Theorem 3.1]{Scott83} and classical work on spherical 3-manifolds (see \cite[p. 113]{Orlik72}).
\end{proof}

We conclude our section of preparations with the following theorem that was recently proved by Huber \cite{Huber}.

\begin{theorem}[Huber]\label{thm:huber}
Let $L(p_1,q_1)$ and $L(p_2,q_2)$ be lens spaces. If $L(p_1,q_1)\leq L(p_2,q_2)$, then one of the following holds:
\begin{enumerate}
\item the lens spaces are homeomorphic,
\item there exists an $n\geq 2$ with $L(p_1,q_1)\cong L(n,1)$ and $p_2/q_2\in \mathcal{F}_n$, where 
\[ \mathcal{F}_n\,:=\, \big\{ \tmfrac{nm^2}{nmk+1}\,\big|\, m>k>0, \operatorname{gcd}(m,k)=1\}.\]
\item $L(p_1,q_1)\cong S^3$.
\end{enumerate}
\end{theorem}

\section{Proof of Theorem \ref{thm: main theorem-2}}
We first provide the proof of Theorem \ref{thm: main theorem-2}, since the statement of 
Theorem \ref{thm: main theorem-2} gives us also most of 
Theorem \ref{thm: main theorem}.

\begin{proof}[Proof of Theorem \ref{thm: main theorem-2}]
Let $Y_0$ and $Y_1$ be  aspherical, closed, connected, oriented 3-manifolds. We assume that $Y_0\geq Y_1$ and that $Y_1\geq Y_0$. We need to show that there exists an orientation-preserving homeomorphism from $Y_0$ to $Y_1$.

Let $W$ be a ribbon $\Q$-homology cobordism from $Y_0$ to $Y_1$.
We denote by $\iota_0\colon Y_0\to W$ and $\iota_1\colon Y_1\to W$ the obvious inclusion maps.  Since we also assume that $Y_1\geq Y_0$  we obtain from Corollary~\ref{cor:map-on-pi1}  that $(\iota_{1})_*\colon \pi_1(Y_1)\to \pi_1(W)$ is an isomorphism. Furthermore by Proposition~\ref{prop:basics-ribbon-cobordism} (2) we know that $(\iota_{0})_*\colon \pi_1(Y_0)\to \pi_1(W)$ is a monomorphism.
We write $\alpha:=((\iota_{1})_*)^{-1}\circ (\iota_{0})_*\colon \pi_1(Y_0) \to \pi_1(Y_1)$.

Since $Y_1$ is aspherical  there exists a map $f\colon Y_0\to Y_1$ with $f_*=\alpha\colon \pi_1(Y_0)\to \pi_1(Y_1)$. 
We will show that this  map $f$ has degree equal to $\pm 1$. For this we consider the following diagram, which is commutative by Lemma \ref{lemma: commutative diagram}:
\begin{equation}\label{diag: commutative diagram}
\begin{tikzcd}
	H_3(Y_0;\Z) \arrow{d}[swap]{(j_0)_*}{\cong} \arrow{r}{(\iota_0)_*}[swap]{\cong} & H_3(W;\Z)  \arrow{d}{(j_W)_*} & H_3(Y_1;\Z) \arrow{l}{\cong}[swap]{(\iota_1)_*}  \arrow{d}{(j_1)_*}[swap]{\cong} \\
	H_3(\pi_1(Y_0);\Z) \arrow{r}{(\iota_0)_*} & H_3(\pi_1(W);\Z) & H_3(\pi_1(Y_1);\Z). \arrow{l}[swap]{(\iota_1)_*}{\cong} 
\end{tikzcd} 
\end{equation}
Here, $j_0 \colon Y_0 \hookrightarrow K(\pi_1(Y_0),1)$ is the natural map of $Y_0$ to the Eilenberg-MacLane space of its fundamental group, and likewise $j_1$ and $j_W$ are defined.

It follows from our hypothesis that $W$ is a $\Q$-homology cobordism
and Proposition~\ref{prop:basics-ribbon-cobordism-h} that the  horizontal maps in the first line of this diagram are isomorphisms.
By hypothesis  $Y_0$ and $Y_1$ are already aspherical, i.e.\ they are Eilenberg-MacLane spaces. Therefore, the maps $(j_0)_*$ and $(j_1)_*$ are isomorphisms. 

Now by commutativity of the right square in the diagram (\ref{diag: commutative diagram}), the vertical map $(j_W)_*$ induced by the inclusion map $j_W$ must be an isomorphism. Therefore, by commutativity of the left square in this diagram, we conclude that the map $(\iota_0)_* \colon H_3(\pi_1(Y_0);\Z) \to H_3(\pi_1(W);\Z)$ is an isomorphism. 

As in the proof of Lemma \ref{lemma: commutative diagram} above, we denote by $K(\varphi)\colon K(G,1) \to K(H,1)$ the map induced on Eilenberg-MacLane spaces by a group homomorphism $\varphi\colon G \to H$. In our situation, we have the two maps $f\colon Y_0 \to Y_1$ and $K((\iota_1)_*^{-1}) \circ K((\iota_0)_*)\colon Y_0 \to Y_1$, where we identify $Y_0$ and $Y_1$ with Eilenberg-MacLane spaces of their respective fundamental group. Both induce the same map at the level of fundamental groups. Since $Y_1$ is aspherical, these maps are homotopic by Whitehead's theorem. Therefore, they induce the same map on homology. By the above observation that the bottom maps in diagram (\ref{diag: commutative diagram}) are isomorphisms, we conclude that $f_* \colon H_3(Y_0) \to H_3(Y_1)$ is an isomorphism, and therefore $f$ has degree $\pm 1$. 

By a standard argument a map $f\colon Y_0\to Y_1$ of degree $\pm 1$ induces an epimorphism of fundamental groups. Since we already know that $f_*=\alpha$ is a monomorphism we see that $f_*\colon \pi_1(Y_0)\to \pi_1(Y_1)$ is an isomorphism. Thus it follows from 
Theorem~\ref{thm:iso-on-pi1-implies-homeo} (2) that there exists a homeomorphism $g\colon Y_0\to Y_1$ that induces $f_*=\alpha$. It remains to show that $g$ is orientation-preserving. To do so we consider the  commutative diagram:
\begin{equation}\label{diag: commutative diagram-2}
\begin{tikzcd}
	H_3(Y_0;\Z) \arrow{d}[swap]{(j_0)_*}{\cong} \arrow{r}{(\iota_0)_*}[swap]{\cong} & H_3(W;\Z)  \arrow{d}{(j_W)_*} & H_3(Y_1;\Z) \arrow{l}{\cong}[swap]{(\iota_1)_*}  \arrow{d}{(j_1)_*}[swap]{\cong} \\
	H_3(\pi_1(Y_0);\Z) \arrow{r}{(\iota_0)_*}
 \arrow{d}[swap]{\operatorname{id}}
 & H_3(\pi_1(W);\Z) & H_3(\pi_1(Y_1);\Z) \arrow{l}[swap]{(\iota_1)_*}{\cong}
 \arrow{d}[swap]{\operatorname{id}}
 \\
	H_3(\pi_1(Y_0);\Z) \arrow{rr}{\alpha=f_*=g_*} & & H_3(\pi_1(Y_1);\Z). \\
	H_3(Y_0;\Z) \arrow{u}[swap]{(j_0)_*}{\cong} \arrow{rr}{g_*}[swap]{\cong} &  & H_3(Y_1;\Z). \arrow{u}{(j_1)_*}[swap]{\cong} 
\end{tikzcd} 
\end{equation}
We already know that the top part of the diagram (\ref{diag: commutative diagram-2}) commutes. The center part of diagram (\ref{diag: commutative diagram-2}) commutes since $\alpha=((\iota_{1})_*)^{-1}\circ (\iota_{0})_*$. The bottom part of diagram (\ref{diag: commutative diagram-2}) commutes again 
by Lemma~\ref{lemma: commutative diagram}. Finally by design we have $\alpha=f_*=g_*$. 

By  Proposition~\ref{prop:basics-ribbon-cobordism-h}  we know that  $(\iota_0)_*([Y_0])=(\iota_1)_*([Y_1])\in H_3(W;\Z)$.  But by the above this implies that 
$g_*([Y_0])=[Y_1]$. 
\end{proof}

\section{Proof of Theorem \ref{thm: main theorem}}

\begin{proof}[Proof of Theorem \ref{thm: main theorem}]
Let $Y_0$ and $Y_1$ be irreducible closed connected oriented 3-manifolds.
We assume that $Y_0\geq Y_1$ and that $Y_1\geq Y_0$. We need to show that 
$Y_0$ and $Y_1$ are homeomorphic.
If $\pi_1(Y_0)$ and $\pi_1(Y_1)$ are infinite, it follows by a standard argument, using the Sphere Theorem, that $Y_1$ is aspherical, see \cite[p.~48]{AFW}.
Thus we see that the statement follows from Theorem \ref{thm: main theorem-2}.

Therefore it suffices to consider the case that  $\pi_1(Y_0)$ or $\pi_1(Y_1)$ is finite. We assume that $Y_0\geq Y_1$ and that $Y_1\geq Y_0$. We need to show that 
$Y_0$ and $Y_1$ are homeomorphic.  By symmetry we can assume that $\pi_1(Y_1)$ is finite. 

Let $W_0$ be a ribbon $\Q$-homology cobordism from $Y_0$ to $Y_1$
and let $W_1$ be a ribbon $\Q$-homology cobordism from $Y_1$ to $Y_0$.
By Corollary~\ref{cor:map-on-pi1} we know that the inclusion induced maps
 $ \pi_1(Y_1)\to \pi_1(W_0)$ and
 $ \pi_1(Y_0)\to \pi_1(W_1)$  are isomorphisms 
and 
by Proposition~\ref{prop:basics-ribbon-cobordism}
we know that the inclusion induced maps  $ \pi_1(Y_0)\to \pi_1(W_0)$ 
and  $\pi_1(Y_1)\to \pi_1(W_1)$ are monomorphisms.

It follows that $|\pi_1(Y_1)|\leq |\pi_1(W_1)|= |\pi_1(Y_0)|\leq |\pi_1(W_0)|=|\pi_1(Y_1)|$.
Since $\pi_1(Y_1)$ is finite we see that we have equalities throughout and we see that all the inclusion induced maps are isomorphisms. In particular we see that $\pi_1(Y_0)\cong \pi_1(Y_1)$.

First we consider the case that $\pi_1(Y_0)$, and thus also $\pi_1(Y_1)$, is not cyclic. 
In this setting it follows from 
Theorem~\ref{thm:iso-on-pi1-implies-homeo} that $Y_0$ is homeomorphic to $Y_1$.

Finally we consider the case that $\pi_1(Y_0)$, and thus also $\pi_1(Y_1)$, is  cyclic. 
 It follows from \cite[p.~25]{AFW} that both $Y_0$ and $Y_1$ are lens spaces. 
But it follows almost immediately from Theorem~\ref{thm:huber}
that if for two lens spaces $Y_0$ and $Y_1$ we have $Y_1\geq Y_0$ and if they have isomorphic fundamental groups, then $Y_0$ and $Y_1$ are homeomorphic. 
\end{proof}

\section{Sketch of proof of Theorem \ref{thm: ribbon to itself}}
In this section we provide a sketch of proof of Theorem \ref{thm: ribbon to itself} and  Corollary \ref{cor:map-on-pi1}, both of which are essentially due to Agol \cite{Agol}, although he has formulated it in the context of knot complements. 

\begin{proof}[Sketch of proof of Theorem \ref{thm: ribbon to itself}]
	This follows almost verbatim in the same way as in all but the last paragraph of \cite[Proof of Theorem 1.2]{Agol}. The only comment to make is that his proof uses residual finiteness of fundamental groups of knot complements. This is due to Hempel \cite{Hempel}, using Thurston's proof of his geometrization conjecture for Haken manifolds. In our situation, we need the fact that all fundamental groups of 3-manifolds are residually finite, and this uses the full geometrization conjecture together with Hempel's result. We will outline this proof for the sake of completeness. 
	
		Suppose that $W$ is a rational ribbon homology cobordism from $Y_-$ to $Y_+$, where both $Y_+$ and $Y_-$ are homeomorphic to $Y$. For a finitely presented group $\pi$, Agol considers the representation variety $R_N(\pi) = \Hom(\pi,SO(N))$, for some $N\geq 1$, and in the case of a path-connected topological space $X$, he defines $R_N(X):= R_N(\pi_1(X))$. This representation variety is a real algebraic set. Since the inclusion $\iota_+\colon Y_+ \to W$ defines a surjection at the level of fundamental groups, we obtain an injection $\iota_+^*\colon R_N(W) \to R_N(Y_+)$ by precomposition of representations with $(\iota_+)_*\colon \pi_1(Y_+) \to \pi_1(W)$. In fact, since there is a presentation of $\pi_1(W)$ obtained by one from $\pi_1(Y_+)$ by only possibly adding relations, but no generators, one can realize $R_N(W)$ as an algebraic subset of $R_N(Y_+)$, $R_N(W)  \subseteq R_N(Y_+) $. 
		
		On the other hand, the inclusion $\iota_-\colon Y_- \to W$ induces an injection of fundamental groups, and Daemi, Lidman, Vela-Vick, and Wong have shown that the induced map $(\iota_-)^*\colon R_N(W) \to R_N(Y_-)$ is surjective, see \cite[Proposition 2.1]{DLVW}. Both of these results build on work of Gerstenhaber and Rothaus, the first statement uses \cite[Theorem 2]{Gerstenhaber-Rothaus_1962}, using the residual finiteness of $\pi_1(Y_-)$, and the second statement builds on \cite[Theorem 1]{Gerstenhaber-Rothaus_1962}, where it is essential that the Lie group that Agol considers, the group $SO(N)$, is compact. 
		
		At this stage, $R_N(Y_-)$ and $R_N(Y_+)$ have a priori been considered using different presentations, but a sequence of Tietze moves between the presentations induces a polynomial isomorphism $\phi\colon R_N(Y_-) \to R_N(Y_+)$ between these. Together with the statement in the last paragraph, one obtains a surjective polynomial map $\phi \circ (\iota_-)^*\colon R_N(W) \to R_N(Y_+)$. At this stage Agol uses the following algebraic geometric lemma (\cite[Lemma A.2]{Agol}): If $X$ and $Z$ are real algebraic sets, with $X \subseteq Z$, and if there is a surjective polynomial map $\varphi\colon X \to Z$, then $X=Z$. Applied to our problem, this implies that $R_N(Y_+) = R_N(W)$, induced by the inclusion $(\iota_+)^*\colon R_N(W) \to R_N(Y_+)$. Finally, by using residual finiteness of $\pi_1(Y_+)$ again, and by the fact that any finite group embeds into some $SO(N)$ for sufficiently large $N$, we conclude as in Agol's situation that $(\iota_+)_*\colon \pi_1(Y_+) \to \pi_1(W)$ is an isomorphism, using $R_N(Y_+) = R_N(W)$.  
		\end{proof}

\begin{proof}[Proof of Corollary \ref{cor:map-on-pi1}]
	 Suppose that $W_0$ is a ribbon $\Q$-homology cobordism from $Y_0$ to $Y_1$, and that $W_1$ is a ribbon $\Q$-homology cobordism from $Y_1$ to $Y_0$. We denote by $\iota_{10}\colon Y_1 \to W_0$ the natural inclusion. We form a $\Q$-homology ribbon cobordism from $Y_-:= Y_1$ to $Y_+:=Y_1$ by gluing $W_0$ and $W_1$ along $Y_0$, and we denote by $\iota_+\colon Y_+ \to W$ the natural inclusion. Finally, we denote by $j\colon W_0 \to W$ the natural inclusion. Then clearly $\iota_+ = j \circ \iota_{10}$. This induces a commutative diagram between fundamental groups
\begin{equation}\label{diag: commutative diagram_pi1}
\begin{tikzcd}
	\pi_1(Y_1) \arrow{d}[swap]{(\iota_{10})_*} \arrow{dr}{(\iota_+)_*} &  \\
	\pi_1(W_0) \arrow{r}{j_*} & \pi_1(W).
\end{tikzcd} 
\end{equation}
By Theorem \ref{thm: ribbon to itself}, the map $(\iota_+)_*$ is an isomorphism. 
Since $W_0$ is a ribbon $\Q$-homology cobordism from $Y_0$ to $Y_1$
it follows from Proposition~\ref{prop:basics-ribbon-cobordism} (1) 
that the map $(\iota_{10})_*$ is a surjection. By commutativity of the diagram we conclude that $(\iota_{10})_*$ is a monomorphism.
Thus in summary it is in fact an isomorphism. 
\end{proof}

\bibliography{References}

\providecommand{\bysame}{\leavevmode\hbox to3em{\hrulefill}\thinspace}
\providecommand{\MR}{\relax\ifhmode\unskip\space\fi MR }
\providecommand{\MRhref}[2]{%
  \href{http://www.ams.org/mathscinet-getitem?mr=#1}{#2}
}
\providecommand{\href}[2]{#2}
\begin{thebibliography}{10}

\bibitem{Agol}
Ian Agol, \emph{Ribbon concordance of knots is a partial ordering},
  arXiv:2201.03626, 2012.

\bibitem{AFW}
Matthias Aschenbrenner, Stefan Friedl, and Henry Wilton, \emph{3-manifold
  groups}, EMS Series of Lectures in Mathematics, European Mathematical Society
  (EMS), Z\"{u}rich, 2015. \MR{3444187}

\bibitem{DLVW}
Alikbar Daemi, Tye Lidman, David~Shea Vela-Vick, and C.-M.~Michael Wong,
  \emph{Ribbon homology cobordisms}, arXiv:1904.09721, 2019.

\bibitem{Gerstenhaber-Rothaus_1962}
Murray Gerstenhaber and Oscar~S. Rothaus, \emph{The solution of sets of
  equations in groups}, Proc. Nat. Acad. Sci. U.S.A. \textbf{48} (1962),
  1531--1533. \MR{166296}

\bibitem{Gordon}
Cameron~McA. Gordon, \emph{Ribbon concordance of knots in the {$3$}-sphere},
  Math. Ann. \textbf{257} (1981), no.~2, 157--170. \MR{634459}

\bibitem{Hempel}
John Hempel, \emph{Residual finiteness for {$3$}-manifolds}, Combinatorial
  group theory and topology ({A}lta, {U}tah, 1984), Ann. of Math. Stud., vol.
  111, Princeton Univ. Press, Princeton, NJ, 1987, pp.~379--396. \MR{895623}

\bibitem{Huber}
Marius Huber, \emph{Ribbon cobordisms between lens spaces}, Pac. J. Math.
  \textbf{315} (2021), no.~1, 111--128.

\bibitem{Orlik72}
Peter Orlik, \emph{Seifert manifolds}, Lect. Notes Math., vol. 291, Springer,
  Cham, 1972.

\bibitem{Scott83}
Peter Scott, \emph{There are no fake {Seifert} fibre spaces with infinite
  {{\(\pi_1\)}}}, Ann. Math. (2) \textbf{117} (1983), 35--70.

\bibitem{Waldhausen68}
Friedhelm Waldhausen, \emph{On irreducible 3-manifolds which are sufficiently
  large}, Ann. Math. (2) \textbf{87} (1968), 56--88.

\end{thebibliography}
\bibliographystyle{amsplain}
\end{document}